\documentclass{amsart} 

\usepackage[paperheight=11in, 
    paperwidth=8.5in,
    outer=1.35in,
    inner=1.35in,
    bottom=1.5in,
    top=1.5in]{geometry} 

\usepackage{amsmath, amsxtra, amsthm, amsfonts, amssymb}
\usepackage{mathrsfs}
\usepackage{graphicx}
\usepackage{url}
\usepackage{color}
\usepackage{bbm}
\usepackage{tikz-cd}

\usepackage[T1]{fontenc}
\usepackage{enumitem}

 \usepackage[hidelinks,backref=page]{hyperref} 
 \hypersetup{
    colorlinks,
    citecolor=magenta,
    filecolor=magenta,
    linkcolor=blue,
    urlcolor=black
}

\usepackage{cleveref}

\theoremstyle{definition}
\newtheorem{thm}{Theorem}[section]
\newtheorem{cor}[thm]{Corollary}
\newtheorem{lem}[thm]{Lemma}
\newtheorem{prop}[thm]{Proposition}
\newtheorem{defn}[thm]{Definition}

\newtheorem{ques}[thm]{Question}

\newtheorem{rem}[thm]{Remark}

\newcommand{\RR}{\mathbb R}

\newcommand{\PP}{\mathbb P}
\newcommand{\ZZ}{\mathbb Z}
\newcommand{\M}{\mathrm{M}}
\newcommand{\be}{\mathbf{e}}
\newcommand{\cS}{\mathcal S}
\newcommand{\cQ}{\mathcal Q}
\newcommand{\kk}{\mathbbm k}

\renewcommand{\emptyset}{\varnothing}
\newcommand{\bwedge}{\textstyle{\bigwedge}\displaystyle}

\title{Cohomologies of tautological bundles of matroids}
\author{Christopher Eur}

\begin{document}

\maketitle

\vspace{-10pt}

\begin{abstract}
Tautological bundles of realizations of matroids were introduced in \cite{BEST23} as a unifying geometric model for studying matroids. We compute the cohomologies of exterior and symmetric powers of these vector bundles, and show that they depend only on the matroid of the realization.  As an application, we show that the log canonical bundle of a wonderful compactification of a hyperplane arrangement complement, in particular the moduli space of pointed rational curves $\overline{{M}}_{0,n}$, has vanishing higher cohomologies.
\end{abstract}


\section{Introduction}

Let $E = \{1, \ldots, n\}$ be a finite set, and $\kk$ an algebraically closed field.
Let $L\subseteq \kk^E$ be an $r$-dimensional linear subspace.
The \emph{matroid} $\M$ of $L$ is the data of the set
\[
\{B \subseteq E : \text{the composition $L\hookrightarrow \kk^E \twoheadrightarrow \kk^B$ is an isomorphism}\}
\]
called the set of \emph{bases} of $\M$.  We say that $L$ \emph{realizes} the matroid $\M$.
We will explain notions from matroid theory as necessary, and refer to \cite{Oxl11} for a general background.

\smallskip
In \cite{BEST23}, Berget, Spink, Tseng, and the author introduced \emph{tautological bundles} of realizations of matroids as a new geometric model for studying matroids.
Let us recall the construction.  Let $T$ be the algebraic torus $(\kk^*)^E$, which acts standardly on $\kk^E$ via $(t_1, \dots, t_n) \cdot (x_1, \dots, x_n) = (t_1x_1, \dots, t_nx_n)$.
Let $\PP T = T/\kk^*$ be its projectivization, i.e.\ its quotient by the diagonal.  The image of $t\in T$ in $\PP T$ is denoted $\overline t$.
The \emph{permutohedral variety} $X_E$ (Definition~\ref{defn:perm}) is a smooth projective toric variety with the open dense torus $\PP T$,
considered here as a $T$-variety.
Let $\mathcal O_{X_E}^{\oplus E}$ be the trivial vector bundle $X_E \times \kk^E$ whose $T$-equivariant structure is given by the \emph{inverse} action of $T$ on $\kk^E$, i.e.\ $t \cdot_{\text{inv}} x = t^{-1} x$.

\begin{defn}
The \emph{tautological subbundle} and \emph{quotient bundle} of $L\subseteq \kk^E$ are the $T$-equivariant vector bundles $\cS_L$ and $\cQ_L$ (respectively) on $X_E$ defined by
\begin{align*}
\cS_L =& \text{ the $T$-equivariant subbundle of $\mathcal O_{X_E}^{\oplus E}$ whose fiber over $\overline t\in \PP T \subset X_E$ is $t^{-1}L$, and}\\
\cQ_L =& \text{ the $T$-equivariant quotient bundle of $\mathcal O_{X_E}^{\oplus E}$ whose fiber over $\overline t\in \PP T \subset X_E$ is $\kk^E/t^{-1}L$.}
\end{align*}
\end{defn}

For well-definedness, see \cite[Section 3.1]{BEST23}.  The authors of \cite{BEST23} showed that the $K$-classes $[\cS_L]$ and $[\cQ_L]$ of these vector bundles depend only on the matroid $\M$.  Moreover, by studying the Chern classes and sheaf Euler characteristics of the tautological bundles, both of which depend only on the $K$-class, they were able to unify, recover, and extend various recent developments in algebro-geometric studies of matroids.

Here, we ask:  How do the sheaf cohomologies of $\cS_L$ and $\cQ_L$ depend on the matroid $\M$?

\smallskip
Our main results are as follows.  We say that an element $e\in E$ is a \emph{coloop} (resp.\ \emph{loop}) of $L$ if the decomposition $\kk^{E\setminus e} \oplus \kk^{\{e\}}$ of $\kk^E$ decomposes $L$ into $L' \oplus \kk$ (resp.\ $L' \oplus \{0\}$) for some $L' \subseteq \kk^{E\setminus n}$, or equivalently, if every basis of $\M$ includes (resp.\ excludes) $e$.

\begin{thm}\label{thm:extcohom}
Exterior powers of $\cS_L$ and $\cQ_L$ have vanishing higher cohomologies, i.e.
\[
H^i(\bwedge^p \cS_L) = 0 \quad\text{and}\quad H^i(\bwedge^p \cQ_L)   =0 \quad \text{for all $i > 0$ and $p\geq 0$},
\]
and we have
\[
\sum_{p \geq 0} \dim H^0(\bwedge^p \cS_L)u^p = (u+1)^{|\text{coloops}(\M)|}   \quad\text{and}\quad
\sum_{p\geq 0} \dim H^0(\bwedge^p \cQ_L) u^p = \sum_{\substack{S\subseteq E \\ \text{$S$ contains }\\ \text{a basis of $\M$}}}u^{|E|-|S|}
\]
where $u$ is a formal variable.
\end{thm}

\begin{thm}\label{thm:symcohom}
The symmetric powers of $\cQ_L$ have vanishing higher cohomologies, i.e.\
\[
H^i(\operatorname{Sym}^p \cQ_L) = 0  \quad\text{for all $i > 0$ and $p\geq 0$},
\] 
and we have
\[
\sum_{p \geq 0} \dim H^0(\operatorname{Sym}^p\cQ_L) u^p = \Big(\frac{1}{1-u}\Big)^{|E| - |\text{coloops}(\M)|}
\]
where $u$ is a formal variable.
\end{thm}

In particular, the theorems imply that the cohomologies of exterior and symmetric powers of $\cQ_L$, and those of exterior powers of $\cS_L$, depend only on the matroid $\M$ that $L$ realizes.
One may contrast this to the fact that exterior and symmetric powers of a realization $L$ of $\M$ are not in general determined by the matroid \cite{LV81, Mas81}.

\begin{rem}\label{rem:crem}
Similar results for the dual vector bundles $\cS_L^\vee$ and $\cQ_L^\vee$ can be obtained as follows.
Using the standard dot product on $\kk^E$, let us identify $\kk^E \simeq (\kk^E)^\vee$, so the trivial bundle $(\mathcal O_{X_E}^{\oplus E})^\vee$ is identified with $X_E \times \kk^E$ where $T$ now acts standardly on $\kk^E$.   Denoting $L^\perp$ for the space $(\kk^E/L)^\vee$ considered as a subspace of $\kk^E \simeq (\kk^E)^\vee$, we identify $\cQ_L^\vee$ as the subbundle of $(\mathcal O_{X_E}^{\oplus})^\vee$ whose fiber over $\overline t\in \PP T \subset X_E$ is $t L^{\perp}$.
The permutohedral variety $X_E$ has the \emph{Cremona} involution $\operatorname{crem}: X_E \overset\sim\to X_E$, induced by sending $\overline t\in \PP T$ to $\overline{t^{-1}}$ (see for instance \cite[Section 2.6]{BEST23}).
Our description of $\cQ_L^\vee$ above shows that $\cQ_L^\vee \simeq \operatorname{crem}\cS_{L^\perp}$, and similarly one has $\cS_L^\vee \simeq \operatorname{crem} \cQ_{L^\perp}$.  In particular, symmetric and exterior powers of $\cS_L^\vee$ have vanishing higher cohomologies.
\end{rem}

We prove Theorems~\ref{thm:extcohom} and \ref{thm:symcohom} by establishing a ``deletion-contraction'' property for the tautological bundles, which we now describe.
For a subset $S\subseteq E$, we denote
\begin{align*}
L\backslash S &= \text{ the image of $L$ under the projection $\kk^E \twoheadrightarrow \kk^{E\setminus S}$, and}\\
L/S &= L \cap (\kk^{E\setminus S} \times \{0\}^S), \text{ considered as a subspace of $\kk^{E\setminus S}$}.
\end{align*}
When $S = \{e\}$ is a singleton we often omit the brackets to write $L\backslash e$ and $L/e$, called the \emph{deletion} and \emph{contraction} of $L$ by $e$, respectively.
For an element of $E$, say $n\in E$ for concreteness, there is a natural projection map $f: X_E \to X_{E\setminus n}$ (Definition~\ref{defn:delmap}).
We show the following deletion-contraction property for the pushforward $f_*$ of the tautological bundles of $L$.

\begin{thm}\label{thm:extdelcont}
For all $p \geq 0$, we have
\begin{align*}
R^if_*(\bwedge^p \cS_L) &= 0 \quad\text{for all $i>0$,} \quad\text{and}\\ f_*(\bwedge^p \cS_L) &= \begin{cases}
\bwedge^p(\cS_{L/n} \oplus \mathcal O_{X_{E\setminus n}}) & \text{if $n$ is a coloop in $\M$}\\
\bwedge^p \cS_{L/n} & \text{if $n$ is not a coloop in $\M$}.
\end{cases}
\end{align*}
Similarly, for all $p \geq 0$, we have
\begin{align*}
R^if_*(\bwedge^p \cQ_L) &= 0 \quad\text{for all $i>0$,} \quad\text{and}\\ f_*(\bwedge^p \cQ_L) &= \begin{cases}
\bwedge^p(\cQ_{L\backslash n} \oplus \mathcal O_{X_{E\setminus n}}) & \text{if $n$ is a loop in $\M$}\\
\bwedge^p\cQ_{L/n} & \text{if $n$ is a coloop in $\M$}\\
\bwedge^p\cQ_{L/n} \oplus \bwedge^{p-1}\cQ_{L\setminus n} & \text{if $n$ is neither a loop nor a coloop in $\M$}.
\end{cases}
\end{align*}
\end{thm}

\begin{thm}\label{thm:symdelcont}
For all $p\geq 0$, we have
\[
f_*\operatorname{Sym}^p\cS_L = \begin{cases}
\operatorname{Sym}^p(\cS_{L/n} \oplus \mathcal O_{X_{E\setminus n}}) &\text{if $n$ a coloop}\\
\operatorname{Sym}^p\cS_{L/n} &\text{if $n$ not a coloop}.
\end{cases}
\]
For all $p \geq 0$, we have
\begin{align*}
R^i f_*\operatorname{Sym}^p\cQ_L &= 0 \quad\text{for all $i > 0$,}\quad \text{and}\\
f_* \operatorname{Sym}^p\cQ_L &= \begin{cases}
\operatorname{Sym}^p \cQ_{L/n} & \text{if $n$ a coloop}\\
\operatorname{Sym}^p(\cQ_{L/n} \oplus \mathcal O_{X_{E\setminus n}}) &\text{if $n$ not a coloop}.
\end{cases}
\end{align*}
\end{thm}

\begin{proof}[Proof of Theorems~\ref{thm:extcohom} and \ref{thm:symcohom}, assuming Theorems~\ref{thm:extdelcont} and \ref{thm:symdelcont}]
We induct on the cardinality of $E$, where the statements in the base case $|E| = 1$ are straightforward since $X_E$ is a point in that case.
When $|E|>1$, for all $p\geq 0$, the Leray spectral sequence $E_2^{a,b} = H^a(X_{E\setminus n}, R^b f_* (\bwedge^p\cS_L))$ satisfies $E_2^{a,b} = 0$ for all $b >0$ by Theorem~\ref{thm:extdelcont}, so that
\[
H^i(X_E, \bwedge^p\cS_L) \simeq H^i(X_{E\setminus n}, f_* (\bwedge^p\cS_L)) \quad\text{for all $i\geq 0$}.
\]
Similar statements hold for $\bwedge^p\cQ_L$ and $\operatorname{Sym}^p\cQ_L$ by the same argument.
From the formula for the pushforward $f_*$ of these bundles in Theorems~\ref{thm:extdelcont} and \ref{thm:symdelcont}, we conclude by induction hypothesis the vanishing of higher cohomologies.
Moreover, the formula for $f_*(\bwedge^p \cS_L)$ implies that the polynomial $g(L,u) = \sum_{p \geq 0} \dim H^0(\bwedge^p \cS_L) u^p$ satisfies the relation
\[
g(L,u) =
\begin{cases}
(u+1) \cdot g(L/n, u)& \text{if $n$ is a coloop in $\M$}\\
g(L/n,u) & \text{if $n$ is not a coloop in $\M$},
\end{cases}
\]
hence $g(L,u) = (u+1)^{|\text{coloops}(\M)|}$.
One similarly computes $\sum_{p \geq 0} \dim H^0(\operatorname{Sym}^p \cQ_L) u^p$.
Lastly, the formula for $f_*(\bwedge^p\cQ_L)$ implies that the polynomial $h(L,u) = \sum_{p\geq 0} \dim H^0(\bwedge^p \cQ_L) u^p$ satisfies
\[
h(L,u) =
\begin{cases}
(u+1) \cdot h(L\backslash n, u) & \text{if $n$ is a loop in $\M$}\\
h(L/n, u)& \text{if $n$ is a coloop in $\M$}\\
u \cdot h(L\backslash n, u)+ h(L/n,u) & \text{if $n$ is neither a loop nor a coloop in $\M$}.
\end{cases}
\]
Feeding this into the recipe formula for deletion-contraction invariants \cite[Theorem 2.16]{Wel99} gives $h(L,u) = u^{|E| - r} \mathrm{T}_\M(1, 1+ u^{-1})$ where $\mathrm{T}_\M$ is the \emph{Tutte polynomial} of $\M$, whose corank-nullity description [(2.13), loc.~cit.] gives the desired formula for $h(L,u)$.
\end{proof}

Introduced in \cite{dCP95}, a \emph{wonderful compactification} (Definition~\ref{defn:wndcpt}) of $L\subseteq \kk^E$ is a compactification $W_L$ of $\PP L \cap \PP T$ that served as a key geometric model behind the Hodge theory of matroids \cite{FY04, AHK18}.
Its boundary $\partial W_L = W_L \setminus (\PP L \cap \PP T)$ is a simple normal crossings divisor.
We use Theorem~\ref{thm:extcohom} to deduce the following.

\begin{cor}\label{cor:wndcpt}
The log canonical divisor $K_{W_L} + \partial W_L$ of a wonderful compactification $W_L$ of $L$ has vanishing higher cohomologies, i.e.
\[
H^i(\mathcal O_{W_L}(K_{W_L} + \partial W_L)) = 0 \quad\text{for all $i > 0$,}
\]
and we have
\[
\dim H^0(\mathcal O_{W_L}(K_{W_L} + \partial W_L)) = \sum_{\substack{S\subseteq E \\ \text{$S$ contains }\\ \text{a basis of $\M$}}}(-1)^{|S|-r}.
\]
\end{cor}

The moduli space $\overline{M}_{0,n}$ of pointed rational curves arises as a wonderful compactification of a linear subspace whose matroid is the cyclic matroid of the complete graph on $n-1$ vertices \cite[Section 4.3 Remarks (3)]{dCP95} (see also \cite{Kee92}).
Hence, the corollary in particular implies that the log canonical divisor of $\overline{M}_{0,n}$ has vanishing higher cohomologies, and recovers the classical result that $\dim H^0(\mathcal O_{\overline{{M}}_{0,n}}(K_{\overline{{M}}_{0,n}} + \partial{\overline{{M}}_{0,n}})) = (n-2)!$.

\begin{rem}\label{rem:speyer}
Corollary~\ref{cor:wndcpt} is the ``dual version'' of the following outstanding question in matroid theory due to Speyer about the \emph{anti} log canonical divisor.
Speyer asked whether
\[
(-1)^{r-1} \chi \big( \mathcal O_{W_L}(-K_{W_L} - \partial W_L)\big) \geq 0
\]
for all $L\subseteq \kk^E$ such that its matroid $\M$ is a connected matroid.\footnote{This is an equivalent formulation of the original question, which asked whether \cite[Proposition 3.3]{Spe09} holds over positive characteristic.  We omit the details of the equivalence, which was communicated to the author by David Speyer.}
One can ask more strongly whether $H^i\big( \mathcal O_{W_L}(-K_{W_L} - \partial W_L)\big) = 0$ for all $i<r-1$, which implies the nonnegativity.
Speyer showed that the validity of this nonnegativity implies a bound on the $f$-vectors of matroidal subdivisions \cite{Spe09}.
Over characteristic zero, he proved the nonnegativity via Kawamata--Viehweg vanishing.
\end{rem}

Corollary~\ref{cor:wndcpt} also implies that the cohomologies of the log canonical divisor on a wonderful compactification $W_L$ depends only on the matroid of $L$.
In tropical geometry, for an arbitrary matroid possibly with no realization, instead of the wonderful compactifcation we have its \emph{tropical linear space} \cite{Stu02, AK06, Spe08}, which serves as building blocks of \emph{tropical manifolds}.
With the theory of tropical vector bundles in its infancy, we ask:

\begin{ques}\label{ques:troplin}
Is there a theory of tropical line bundles and their sheaf cohomology on tropical manifolds such that it agrees with Corollary~\ref{cor:wndcpt}?
\end{ques}

Related discussions and questions can be found in Section~\ref{sec:misc}.

\subsection*{Previous works}
When the characteristic of $\kk$ is zero, so that tools like resolution of singularities and Kawamata--Viehweg vanishing are available, parts of the results here have been established in previous literature \cite{BF22, KT09}.  For instance, \cite[Theorem C]{BF22} states that any Schur functor applied to $\cS_L^\vee$ has vanishing higher cohomologies.  The vanishing higher cohomologies of the log canonical divisor of $W_L$ (Corollary~\ref{cor:wndcpt}) is also immediate from Kawamata--Viehweg vanishing when one notes that $\partial W_L$ is big and nef.
The proofs of these previous results crucially depend on characteristic zero methods.
The vanishing statements here are established over fields of arbitrary characteristic by elementary methods.

\subsection*{Organization}
In Section~\ref{sec:perm}, we review permutohedral varieties, and detail the behavior of the projection map $f: X_E \to X_{E\backslash n}$.
In Section~\ref{sec:proof}, after some preparatory computations on $\PP^1$, we prove Theorems~\ref{thm:extdelcont} and \ref{thm:symdelcont}.
In Section~\ref{sec:wndcpt}, we explain the application to wonderful compactifications.
In Section~\ref{sec:misc}, we collect some questions.

\subsection*{Acknowledgements}
The author thanks Andrew Berget, Alex Fink, Dhruv Ranganathan, and David Speyer for helpful conversations, and thanks Matt Larson for helpful conversations and comments on a preliminary draft of the paper.
The author also thanks the referees for helpful suggestions.
The author is supported by US National Science Foundation (DMS-2001854 and DMS-2246518).

\section{Permutohedral varieties}\label{sec:perm}

For a subset $S\subseteq E$, let $\be_S$ be the sum of standard basis vectors $\sum_{i\in S} \be_i \in \RR^E$, and let $\overline\be_S$ be its image in $\RR^E/\RR\be_E$.
For background and conventions for polyhedral geometry and toric geometry, we refer to \cite{Ful93, CLS11}.

\begin{defn}\label{defn:perm}
An \emph{ordered set partition} $\mathscr F$ of $E$ is a sequence $(F_1, \ldots, F_\ell)$ of nonempty subsets of $E$ that partition $E$.
The \emph{permutohedral fan} $\Sigma_E$ is the fan in $\RR^E/\RR\be_E$ consisting of cones
\[
\sigma_{\mathscr F} = \operatorname{cone}\{\overline\be_{F_1}, \overline\be_{F_1\cup F_1}, \ldots, \overline\be_{F_1 \cup \cdots \cup F_\ell}\}
\]
for each ordered set partition $\mathscr F = (F_1, \ldots, F_\ell)$ of $E$.
The \emph{permutohedral variety} $X_E$ is the (smooth projective) toric variety associated to the fan $\Sigma_E$, considered as a rational fan over $\ZZ^E/\ZZ\be_E$.
\end{defn}

We identify the cocharacter lattice of $T = (\kk^*)^E$ with $\ZZ^E$, which identifies the cocharacter lattice of $\PP T$ with $\ZZ^E/\ZZ\be_E$.  This identifies the dense open torus of $X_E$ with $\PP T$, and so we treat $X_E$ as a $T$-variety.
We refer to \cite[Chapter 3]{CLS11} for a background on torus-orbits of toric varieties, and fix the following notations.
For an ordered set partition $\mathscr F = (F_1, \cdots, F_{\ell})$ of $E$, denote by
\begin{itemize}
\item $p_{\mathscr F} = \lim_{t\to 0} \lambda(t)$ the limit point in $X_E$ where $\lambda: \kk^* \to \PP T$ is the one-parameter map of any cocharacter $\lambda\in \ZZ^E / \ZZ\be_E$ in the relative interior $\operatorname{relint}(\sigma_{\mathscr F})$ of $\sigma_{\mathscr F}$,
\item $O_{\mathscr F}$ the $T$-orbit corresponding to $\sigma_{\mathscr F}$, i.e.\ the orbit $T\cdot p_{\mathscr F}$, and
\item $Z_{\mathscr F}$ the closure of the $T$-orbit $O_{\mathscr F}$.
\end{itemize}

We now describe the map $f: X_E \to X_{E\setminus n}$.  First, note that the projection map $\overline f: \RR^E \to \RR^{E\setminus n}$ induces a map of fans $\Sigma_E \to \Sigma_{E\setminus n}$.
We record the following observation, whose verification is straightforward and is omitted.

\begin{prop}\label{prop:fibers}
Let $\mathscr F =(F_1, \dots, F_\ell)$ be an ordered set partition of $E\setminus n$.
The inverse image of the cone $\sigma_{\mathscr F} \in \Sigma_{E\setminus n}$ under the map $\Sigma_E \to \Sigma_{E\setminus n}$ consists of cones in $\Sigma_E$ corresponding to the following two kinds of ordered set partitions of $E$:
\begin{itemize}
\item For $1\leq i \leq \ell+1$, let $\mathscr F^i = (F_1, \dots, F_{i-1}, n, F_i, \dots, F_\ell)$.
\item For $1\leq i \leq \ell$, let $\mathscr F(i) = (F_1, \ldots, F_{i-1}, F_i \cup n, F_{i+1}, \dots, F_\ell)$.
\end{itemize}
Note that $\sigma_{\mathscr F^i} \cap \sigma_{\mathscr F^{i+1}} = \sigma_{\mathscr F(i)}$.
\end{prop}

\begin{defn}\label{defn:delmap}
Let $f: X_E \to X_{E\setminus n}$ be the toric map associated to the map of fans $\Sigma_E \to \Sigma_{E\setminus n}$ induced by the projection map $\overline f: \RR^E \to \RR^{E\setminus n}$.
\end{defn}
 
Translating the polyhedral statement in Proposition~\ref{prop:fibers} to toric geometry gives the following.

\begin{cor}\label{cor:fibers}
The map $f: X_E \to X_{E\setminus n}$ is a flat and projective map whose fibers are chains of rational curves.  More specifically, for any $(t,1)\in T$ where $t\in (\kk^*)^{E\setminus n}$ and an ordered set partition $\mathscr F = (F_1,\dots, F_\ell)$ of $E\setminus n$, we have that the fiber
$
f^{-1}(t \cdot p_{\mathscr F}) = \bigcup_{i=1}^\ell C(t,i)
$
where
\begin{align*}
 C(t, i)  &= \{(t,1)\cdot p_{\mathscr F^i}\} \sqcup \{(t,1) \cdot p_{\mathscr F^{i+1}}\} \sqcup \{(t,t_n)\cdot p_{\mathscr F(i)} : t_n \in \kk^*\}\\
 &= \{0\} \sqcup \{\infty\} \sqcup \kk^* \simeq \PP^1.
\end{align*}
\end{cor}

One may also deduce the first statement of the corollary by noting that $X_E$ is the Losev-Manin space \cite{LM00}, which is a particular Hassett space of rational curves with weighted markings, and that the map $f$ is the universal curve map.

\smallskip
For proving Thereoms~\ref{thm:extdelcont} and \ref{thm:symdelcont}, Corollary~\ref{cor:fibers} primes us to use Grauert's theorem, which we recall here for convenience \cite[Corollary III.12.9]{Har77}:

If $\varphi: X \to Y$ is projective and $\mathcal F\in \operatorname{Coh}(X)$ is flat over $Y$ such that $\dim H^i(X_y, \mathcal F_y)$ is constant over the fibers $X_y = \varphi^{-1}(y)$ of $y\in Y$, then $R^i\varphi_*\mathcal F$ is a vector bundle on $Y$ whose fiber at $y\in Y$ is $H^i(X_y, \mathcal F_y)$.
Note that if $\varphi$ is itself flat, then any vector bundle $\mathcal E$ on $X$ is flat over $Y$.

\medskip
We conclude this section by discussing the behavior of the tautological bundles of $L$ on the fibers of the map $f$.
Let us write $L|S = L\backslash(E\backslash S)$ for a subset $S \subseteq E$.  For an ordered set partition $\mathscr F = (F_1, \ldots, F_\ell)$ of $E$, let $L_{\mathscr F}$ be the linear subspace
\[
 L_{\mathscr F}= L | F_1 \oplus L|(F_1\cup F_2)/F_1 \oplus \cdots \oplus L|(F_1 \cup \cdots \cup F_{\ell-1})/(F_1 \cup \cdots \cup F_{\ell-2}) \oplus L/(F_1 \cup \cdots \cup F_{\ell-1}) 
\]
of $\kk^{F_1} \oplus \kk^{F_2} \oplus \cdots \oplus \kk^{F_\ell} = \kk^E$.  Note that the dimension remains the same ($\dim L_{\mathscr F} = \dim L$), which one may verify by a straightforward induction.
We will also need the following fact, which follows from \cite[Proposition 3.6 and Proposition 5.3]{BEST23}.

\begin{lem}\label{lem:Lbndry}
The restriction $\cS_L|_{Z_{\mathscr F}}$ (resp.\ $\cQ_L|_{Z_{\mathscr F}}$) is the unique $T$-equivariant subbundle (resp.\ quotient bundle) of $(\mathcal O_{X_E}^{\oplus E})|_{Z_{\mathscr F}} = Z_{\mathscr F} \times {\kk}^E$ whose fiber over $p_{\mathscr F}$ is $L_{\mathscr F}$ (resp.\ $\kk^E/L_{\mathscr F}$).
\end{lem}

The tautological bundles of $L$ restricted to a fiber of $f$ are simple in the following sense.

\begin{prop}\label{prop:constant}
Let notations be as in Corollary~\ref{cor:fibers}.
As a subbundle (resp.\ quotient bundle) of the trivial bundle $\mathcal O^{\oplus E}$, the fibers of the restricted bundle $\cS_L|_{f^{-1}(t\cdot p_{\mathscr F})}$ (resp.\ $\cQ_L|_{f^{-1}(t \cdot p_{\mathscr F})}$) are constant along all $\PP^1$-components $C(t,i)$ if $n$ is a loop or a coloop of $L$, and non-constant at exactly one component if $n$ is neither a loop nor a coloop of $L$.
\end{prop}

\begin{proof}
If $n$ is a loop or a coloop of $L$, then $(t,t_n) \cdot L = (t,1) \cdot L$ for all $t_n \in \kk^*$, and for any $S\subseteq E\setminus n$, the element $n$ is again a loop or a coloop of $L\backslash S$ and $L/S$.
Thus, in this case $n$ is a loop or a coloop of $L_{\mathscr F(i)}$ for all $1\leq i\leq \ell$, and so Lemma~\ref{lem:Lbndry} and Corollary~\ref{cor:fibers} together imply that $\cS_L$ and $\cQ_L$ are constant along each component $C(t,i)$.
Suppose now that $n$ is neither a loop nor a coloop of $L$.  We need show that $n$ is neither a loop nor a coloop in $L_{\mathscr F(k)}$ for exactly one $1\leq k \leq \ell$.
For this end, we will use the following property of matroids that follows from their greedy algorithm structure (see \cite[\S1.8]{Oxl11}):

For an ordered partition $\mathscr F' = (F'_1, \dots, F'_\ell)$ of $E$, let $w_{\mathscr F'}: E \to \RR$ be any weighting such that $w_{\mathscr F'}$ is constant on each $F'_i$ and $w(f'_i) > w(f'_j)$ whenever $f'_i\in F'_i$ and $f'_j\in F'_j$ with $i<j$.  Then, the $w_{\mathscr F'}$-maximal bases of the matroid $\M$ of $L$ are the bases of the matroid of $L_{\mathscr F}$.

Now, note that if $n$ is neither a loop nor a coloop in $L_{\mathscr F(k)}$, then it is a coloop in $L_{\mathscr F^k}$ and a loop in $L_{\mathscr F^{k+1}}$ by construction.  In particular, every $w_{\mathscr F^k}$-maximal basis of $\M$ includes $n$, and every $w_{\mathscr F^{k+1}}$-maximal basis of $\M$ excludes $n$.  Thus, every $w_{\mathscr F(i)}$-maximal basis of $\M$ must include $n$ if $i<k$, and must exclude $n$ if $i>k$.
\end{proof}

\section{Proof of Theorems~\ref{thm:extdelcont} and \ref{thm:symdelcont}}\label{sec:proof}

Since $\cS_L$ or $\cQ_L$ along a fiber of the map $f$ is non-constant on at most one $\PP^1$-component of the fiber (Proposition~\ref{prop:constant}), we begin with preparatory observations for vector bundles on $\PP^1$.
Consider $\PP^1 = \{0\} \sqcup \{\infty\} \sqcup \kk^*$ as a $\kk^*$-toric variety, and let $\mathcal O_{\PP^1}^{\oplus E}$ be the trivial vector bundle $\PP^1 \times \kk^E$ where $\kk^*$ acts on $\kk^E$ by $t\cdot (x_1, \ldots, x_{n-1}, x_n) = (x_1, \ldots, x_{n-1}, t^{-1} x_n)$.
We write $\kk^{\{n\}}$ for the last coordinate of $\kk^E$ with the inverse standard action of $\kk^*$.

\begin{defn}\label{defn:tautP1}
For a subspace $L\subseteq \kk^E$, let $\cS_L'$ and $\cQ_L'$ be the $\kk^*$-equivariant sub and quotient bundles of $\mathcal O_{\PP^1}^{\oplus E}$, respectively, fitting into a short exact sequence $0\to \cS_L' \to \mathcal O_{\PP^1}^E \to \cQ_L' \to 0$ such that its fiber over the identity of $\kk^*$ is $0\to L \to \kk^E \to \kk^E/L \to 0$.
\end{defn}

Note that if $n$ is a loop or a coloop of $L$, so that $L = L' \oplus L|n \subseteq \kk^{E\setminus n} \oplus \kk^{\{n\}}$, we have that $\cS'_L \simeq \mathcal O_{\PP^1} \otimes (L' \oplus L|n)$ is a 
trivial bundle, and similarly for $\cQ'_L$.

\begin{lem}\label{lem:P1}
We have short exact sequences
\[
0\to \cS'_{L/n \oplus 0} \to \cS'_L \to \mathcal L_{\cS} \to 0 \quad\text{and}\quad 0\to \mathcal L_\cQ \to \cQ'_L \to \cQ'_{L\backslash n \oplus L|n} \to 0,
\]
where $\mathcal L_\cS$ and $\mathcal L_\cQ$ are $\kk^*$-equivariant line bundles (or zero) defined by
\[
\mathcal L_{\cS}  = \begin{matrix}
\text{the $\kk^*$-equivariant subbundle of}\\
\text{$\mathcal O_{\PP^1} \otimes \big( (\kk^{E\setminus n} / (L/n)) \oplus \kk^{\{n\}}\big)$}\\
\quad\text{whose fiber at the identity is $L/ (L/n \oplus 0)$}\end{matrix}
\simeq  \begin{cases}
0 & \text{if $n$ a loop}\\
\mathcal O_{\PP^1}\otimes \kk^{\{n\}} & \text{if $n$ a coloop}\\
\mathcal O_{\PP^1}(-1) & \text{if $n$ neither}
\end{cases}
\]
and
\[
\mathcal L_{\cQ}  = \begin{matrix}
\text{the $\kk^*$-equivariant quotient bundle of}\\
\text{$\mathcal O_{\PP^1} \otimes  \big( (L\backslash n)/(L/n) \oplus L|n \big) $}\\
\quad\text{whose fiber at the identity is the quotient by $L/(L/n \oplus 0)$}\end{matrix}
\simeq  \begin{cases}
0 & \text{if $n$ a loop}\\
0 & \text{if $n$ a coloop}\\
\mathcal O_{\PP^1}(1) & \text{if $n$ neither}.
\end{cases}
\]
\end{lem}

\begin{proof}
If $n$ is a loop or a coloop, so that $\cS'_L$ and $\cQ'_L$ are trivial bundles, the statements of the lemma are immediate.  Let us now assume that $n$ is neither a loop nor a coloop.

For all $t\in \kk^* \subset \PP^1$, we have $L/n \oplus 0 = (t\cdot L) \cap (\kk^{E\setminus n} \oplus 0) \hookrightarrow t\cdot L$, and at the boundaries we have $\lim_{t\to 0} t\cdot L = L/ n \oplus 0$ and $\lim_{t\to \infty}t\cdot L = L\backslash n \oplus 0$.
Since $L/n \subseteq L\backslash n$, we thus have an injective map of vector bundles $\cS'_{L\oplus 0} \hookrightarrow \cS'_L$.
We hence obtain the following diagram of commuting short exact sequences
\[
\begin{tikzcd}
 & 0 \ar[d] & 0 \ar[d] &  &\\
 0 \ar[r] & \cS'_{L/n \oplus 0} \ar[r,equal]\ar[d]& \mathcal O_{\PP^1} \otimes (L/n \oplus 0) \ar[r] \ar[d]& 0 \ar[d]& \\
 0 \ar[r] & \cS'_L \ar[r] \ar[d]& \mathcal O_{\PP^1} \otimes \kk^E \ar[r]\ar[d] & \cQ'_L \ar[r]\ar[d,equal] & 0\\
0 \ar[r] & \mathcal L_\cS \ar[r] \ar[d]& \mathcal O_{\PP^1} \otimes \big(\kk^{E\setminus n}/(L/n) \oplus \kk^{\{n\}}\big) \ar[r]\ar[d] &\cQ'_L \ar[r]\ar[d] & 0\\
 & 0 & 0 & 0 &
\end{tikzcd}
\]
by starting with the first two rows and then applying the snake lemma.
We have the desired short exact sequence for $\cS'_L$.

Now, over the identity point of $\kk^* \subset \PP^1$, the $\kk^*$-equivariant embedding $\mathcal L_\cS \hookrightarrow \mathcal O_{\PP^1} \otimes \big( (\kk^{E\setminus n} / (L/n)) \oplus \kk^{\{n\}}\big)$ is  $L/(L/n\oplus 0) \hookrightarrow (\kk^{E\setminus n} / (L/n)) \oplus \kk^{\{n\}}$.
Because $(L\backslash n)/(L/n) \oplus L|n$ is the direct sum of the projections of $L/(L/n \oplus 0)$ to the two direct summands of $(\kk^{E\setminus n} / (L/n)) \oplus \kk^{\{n\}}$, we have that 
$\mathcal L_\cS$ in fact embeds in $\mathcal O_{\PP^1} \otimes \big( (L\backslash n)/(L/n) \oplus L|n \big)$.
In other words, $\mathcal L_\cS$ is the pullback of the tautological subbundle of $\PP^1 \simeq \PP\big((L\backslash n)/(L/n) \oplus L|n\big)$ where the isomorphism $\PP^1 \overset\sim\to \PP\big((L\backslash n)/(L/n) \oplus L|n\big)$ is defined by
\[
\kk^*\ni t \mapsto \text{the image in $(L\backslash n)/(L/n) \oplus L|n$ of }t\cdot L / (L/n \oplus 0).
\]
The Euler sequence $0\to \mathcal O_{\PP^1}(-1) \to \mathcal O_{\PP^1}^2 \to \mathcal O_{\PP^1}(1) \to 0$ on $\PP^1$ then becomes
\[
0 \to \mathcal L_\cS \to \mathcal O_{\PP^1} \otimes \big( (L\backslash n)/(L/n) \oplus L|n \big) \to \mathcal L_\cQ \to 0,
\]
which defines the line bundle $\mathcal L_\cQ$, and proves the statements about the isomorphism types of $\mathcal L_\cS$ and $\mathcal L_\cQ$.

Lastly, we have the following commuting diagram of short exact sequences
\[
\begin{tikzcd}
 & 0 \ar[d] & 0 \ar[d] & 0\ar[d] &\\
 0 \ar[r] & \mathcal L_\cS \ar[r]\ar[d,equal]& \mathcal O_{\PP^1} \otimes \big( (L\backslash n)/(L/n) \oplus L|n\big) \ar[r] \ar[d]& \mathcal L_\cQ \ar[d]\ar[r]&0 \\
 0 \ar[r] & \mathcal L_\cS \ar[r] \ar[d]& \mathcal O_{\PP^1} \otimes \big(\kk^{E\setminus n}/(L/n) \oplus \kk^{\{n\}}\big)\ar[r]\ar[d] & \cQ'_L \ar[r]\ar[d] & 0\\
 & 0 \ar[r]& \mathcal O_{\PP^1} \otimes \big(\kk^{E\setminus n}/(L\backslash n) \oplus \kk^{\{n\}}/(L|n) \big) \ar[r,equal] \ar[d]&\cQ'_{L\backslash n \oplus L|n}\ar[r]\ar[d] & 0\\
 &  & 0 & 0 &
\end{tikzcd}
\]
by starting with the first two columns and then applying the snake lemma.  The desired short exact sequence for $\cQ'_L$ follows.
\end{proof}

\begin{rem}
The two short exact sequences in \Cref{lem:P1} split:  For the first sequence, it follows from the  possible isomorphism types of $\mathcal L_\cS$ that $\operatorname{Ext}^1_{\PP^1}(\mathcal L_\cS, \cS'_{L/n \oplus 0}) \simeq H^1(\PP^1, \mathcal L_\cS^\vee \otimes \cS'_{L/n \oplus 0}) = 0$, and similarly for the second sequence.
\end{rem}

The lemma implies the following about the cohomologies of exterior powers of $\cS'_L$ and $\cQ'_L$.

\begin{prop}\label{prop:wedgeH0}
For all $p\geq 0$, we have $H^1(\bigwedge^p \cS'_L) = 0$ and $H^1(\bigwedge^p \cQ'_L) = 0$, and we have natural isomorphisms
\begin{align*}
H^0(\bwedge^p \cS'_L) &\simeq \begin{cases}
\bwedge^p (L/n \oplus \kk) & \text{if $n$ a coloop},\\
\bwedge^p(L/n) & \text{if $n$ not a coloop},
\quad\qquad\text{and}
\end{cases}\\
H^0(\bwedge^p \cQ'_L) &\simeq \begin{cases}
\bwedge^p\big(\kk^{E}/(L/n \oplus 0)\big) & \text{if $n$ a loop}\\
\bwedge^p \big( \kk^E / (L/n \oplus \kk)\big) & \text{if $n$ a coloop}\\
\bwedge^p \big( \kk^{E\setminus n}/(L/n) \big) \oplus \bwedge^{p-1} \big(\kk^{E\setminus n}/(L\backslash n)\big) &\text{if $n$ neither}.
\end{cases}
\end{align*}
\end{prop}

\begin{proof}
By standard multilinear algebra (e.g.\ \cite[Exercise II.5.16]{Har77}), applying exterior powers to the short exact sequences of \Cref{lem:P1} yields short exact sequences
\begin{equation}\tag{$\dagger$}
\begin{split}
0 \to \bwedge^p \cS'_{L/n \oplus 0} \to &\bwedge^p \cS'_L \to \bwedge^{p-1} \cS'_{L/n\oplus 0} \otimes \mathcal L_\cS \to 0 \qquad\text{and}\\
\label{eqn:wedgeSESQ}
0\to \mathcal L_\cQ \otimes \bwedge^{p-1} \cQ'_{L\backslash n \oplus L|n} \to &\bwedge^p \cQ'_L \to \bwedge^p \cQ'_{L\backslash n \oplus L|n} \to 0
\end{split}
\end{equation}
for all $p \geq 0$.
In the resulting long exact sequences of cohomologies, we have $H^1(\bwedge^p \cS'_L)=H^1(\bwedge^p \cQ'_L) = 0$ because of the descriptions of $\mathcal L_\cS$ and $\mathcal L_\cQ$ in \Cref{lem:P1} and $H^1(\mathcal O_{\PP^1}(-1)) = H^1(\mathcal O_{\PP^1}) = H^1(\mathcal O_{\PP^1}(1)) = 0$, keeping in mind that $\cS'_{L/n \oplus 0}$ and $\cQ'_{L\backslash n \oplus L|n}$ are trivial bundles.
As $H^1$'s vanish, note that applying $H^0$ yields short exact sequences of vector spaces.

We now treat the statements about $H^0$.  When $n$ is a loop or a coloop, the desired follows since all vector bundles involved are trivial in such case.  So, assume now that $n$ is neither a loop nor a coloop.
The statement for $H^0(\bwedge^p \cS'_L)$ follows since $H^0(\mathcal O_{\PP^1}(-1)) = 0$.
For $H^0(\bwedge^p \cQ'_L)$, note first that $H^0(\mathcal L) = V$ for any $V \simeq \kk^2$ and $\mathcal L \simeq \mathcal O_{\PP^1}(1)$ such that
\[
0 \to \mathcal O_{\PP^1}(-1) \to  \mathcal O_{\PP^1} \otimes V \to \mathcal L \to 0.
\]
Applying this with $V = (L\backslash n)/(L/n) \oplus L|n$ and $\mathcal L = \mathcal L_\cQ$, we obtain that the short exact sequence from applying $H^0$ to the second sequence in \eqref{eqn:wedgeSESQ} with $p = 1$ is naturally isomorphic to
\[
0\to (L\backslash n)/(L/n) \oplus L|n \to \kk^{E\setminus n}/(L/n) \oplus \kk^{\{n\}} \to \kk^{E\setminus n}/(L\backslash n) \oplus \kk^{\{n\}}/(L|n) \to 0,
\]
(i.e.\ the middle column of the second diagram in the proof of \Cref{lem:P1}) which is the direct sum of two sequences
\[
0\to (L\backslash n)/(L/n) \to \kk^{E\setminus n}/(L/n) \to \kk^{E\setminus n}/(L\backslash n) \to 0 \quad\text{and}\quad 0\to \kk \to \kk \to 0 \to 0.
\]
In general, applying $H^0$ for $p\geq 1$ yields the short exact sequence which is the direct sum of
\begin{align*}
0 \to (L\backslash n)/(L/n) \otimes \bwedge^{p-1}\kk^{E\setminus n}/(L\backslash n) \to &\bwedge^p \kk^{E\setminus n}/(L/n) \to \bwedge^p\kk^{E\setminus n}/(L\backslash n) \to 0 \quad\text{and}\\
0 \to \kk \otimes \bwedge^{p-1}\kk^{E\setminus n}/(L\backslash n) \to &\bwedge^{p-1} \kk^{E\setminus n}/(L\backslash n) \to 0 \to 0.
\end{align*}
The desired statement for $H^0(\bwedge^p\cQ'_L)$ follows.
\end{proof}

We also use \Cref{lem:P1} to deduce the following symmetric powers analogue of \Cref{prop:wedgeH0}.  Note that for any $V \simeq \kk^2$ and $\mathcal L \simeq \mathcal O_{\PP^1}(1)$ fitting into $
0 \to \mathcal O_{\PP^1}(-1) \to  \mathcal O_{\PP^1} \otimes V \to \mathcal L \to 0$, we have natural isomorphisms
\[
H^0(\mathcal L^{\otimes p}) \simeq \operatorname{Sym}^p V \quad\text{and}\quad H^1(\mathcal O_{\PP^1}(-p-2)) \simeq \det V^\vee \otimes \operatorname{Sym}^p V^\vee \quad\text{for all $p\geq 0$}.
\]

\begin{prop}\label{prop:symH}
For all $p \geq 0$, we have a natural isomorphism
\[
H^0(\operatorname{Sym}^p \mathcal S'_L) \simeq
\begin{cases}
\operatorname{Sym}^p(L/n \oplus \kk) & \text{if $n$ a coloop}\\
\operatorname{Sym}^p(L/n) &\text{if $n$ not a coloop}.
\end{cases}
\]
When $n$ is a loop or a coloop, we have $H^1(\operatorname{Sym}^p \cS'_L) = 0$, and when $n$ is neither, we have a filtration $H^1(\operatorname{Sym}^p \cS'_L) = F_0 \supseteq F_1 \supseteq \cdots\supseteq F_{p-2} \supseteq F_{p-1} = 0$ such that 
\[
F_i/F_{i+1} \simeq \big(\det((L\backslash n)/(L/n) \oplus L|n)^\vee\big)^{\otimes p-1-i} \otimes \operatorname{Sym}^{p-2-i}((L\backslash n)/(L/n) \oplus L|n) \otimes \operatorname{Sym}^i(L/n \oplus 0)
\]
for all $0\leq i \leq p-2$.  Similarly, for all $p\geq 0$, we have $H^1(\operatorname{Sym}^p \cQ'_L) = 0$, and we have a natural isomorphism
\[
H^0(\operatorname{Sym}^p\cQ'_L) \simeq
\begin{cases}
\operatorname{Sym}^p \big(\kk^{E\setminus n}/(L/n)\big) &\text{if $n$ a coloop}\\
\operatorname{Sym}^p \big(\kk^{E\setminus n}/(L/n)\oplus \kk\big) & \text{if $n$ not a coloop}.
\end{cases}
\]
\end{prop}

\begin{proof}
When $n$ is a loop or a coloop, the bundles $\cS'_L$ and $\cQ'_L$ are trivial, and the claimed statements follow easily.  Suppose $n$ is neither a loop or coloop now.  For the statements about $\operatorname{Sym}^p \cS'_L$, we first note that the short exact sequence $0\to \cS'_{L/n \oplus 0} \to \cS'_L \to \mathcal L_\cS \to 0$ in \Cref{lem:P1}, along with some multilinear algebra (e.g.\ \cite[Exercise II.5.16]{Har77}), gives a filtration $\operatorname{Sym}^p\cS'_L  = \mathcal F_0 \supseteq \mathcal F_1 \supseteq \cdots \supseteq \mathcal F_p \supseteq \mathcal F_{p+1} = 0$ with $\mathcal F_i / \mathcal F_{i+1} \simeq \operatorname{Sym}^i \cS'_{L/n \oplus 0} \otimes \mathcal L_\cS^{\otimes p-i}$.  In the long exact sequences
\[
0 \to H^0(\mathcal F_{i+1}) \to H^0(\mathcal F_i) \to H^0(\mathcal F_i / \mathcal F_{i+1}) \to H^1(\mathcal F_{i+1}) \to H^1(\mathcal F_i) \to H^1(\mathcal F_i / \mathcal F_{i+1}) \to 0,
\]
we have $H^0(\mathcal F_i / \mathcal F_{i+1}) = 0$ if $i < p$ since $\mathcal L_S \simeq \mathcal O_{\PP^1}(-1)$, and thus we have $H^0(\operatorname{Sym}^p \cS'_L) \simeq H^0(\mathcal F_p / \mathcal F_{p+1}) = \operatorname{Sym}^p (L/n)$.  The filtration for $H^1$ also follows since the $H^1$'s form a short exact sequence for each $i$, and
\begin{align*}
&H^1(\mathcal F_i / \mathcal F_{i+1}) \\
&\simeq \det((L\backslash n)/(L/n) \oplus L|n)^\vee \otimes \operatorname{Sym}^{p-2-i}((L\backslash n)/(L/n) \oplus L|n)^\vee \otimes \operatorname{Sym}^i(L/n \oplus 0)\\
&\simeq \big(\det((L\backslash n)/(L/n) \oplus L|n)^\vee\big)^{\otimes p-1-i} \otimes \operatorname{Sym}^{p-2-i}((L\backslash n)/(L/n) \oplus L|n) \otimes \operatorname{Sym}^i(L/n \oplus 0).
\end{align*}
For the statements about $\operatorname{Sym}^p\cQ'_L$, we similarly have from $0\to \mathcal L_\cQ \to \cQ'_L \to \cQ'_{L\backslash n\oplus L|n} \to 0$ a filtration
$\operatorname{Sym}^p\cQ'_L  = \mathcal F_0 \supseteq \mathcal F_1 \supseteq \cdots \supseteq \mathcal F_p \supseteq \mathcal F_{p+1} = 0$ with $\mathcal F_i / \mathcal F_{i+1} \simeq \mathcal L_\cQ^{\otimes i}\otimes  \operatorname{Sym}^{p-i} \cQ'_{L/n \oplus L|n}$.  Note that $\mathcal L_\cQ \simeq \mathcal O_{\PP^1}(1)$, so all $H^1(\mathcal F_i / \mathcal F_{i+1})$ vanish, and hence $H^1(\mathcal F_i) = 0$ for all $i$ as well.  Now, the resulting short exact sequences
$
0 \to H^0(\mathcal F_{i+1}) \to H^0(\mathcal F_i) \to H^0(\mathcal F_i / \mathcal F_{i+1}) \to 0
$
give a filtration of $H^0(\operatorname{Sym}^p\cQ'_L)$ whose successive quotients are
\[\operatorname{Sym}^i((L\backslash n)/(L/n) \oplus L|n) \otimes \operatorname{Sym}^{p-i}(\kk^{E\setminus n}/(L\backslash n) \oplus \kk/(L|n)).
\]
This is exactly the filtration of $\operatorname{Sym}^p(\kk^{E\setminus n}/(L/n) \oplus \kk)$ arising from the short exact sequence $0\to (L\backslash n)/(L/n) \oplus L|n \to \kk^{E\setminus n}/(L/n) \oplus \kk \to \kk^{E\setminus n}/(L\backslash n) \oplus \kk/(L|n) \to 0$.
\end{proof}


We are now ready to prove the main theorems.

\begin{proof}[Proof of Theorems~\ref{thm:extdelcont} and \ref{thm:symdelcont}]
We first prove the statement for $\bwedge^p\cS_L$.  The statements for $\bwedge^p\cQ_L$, $\operatorname{Sym}^p\cQ_L$, and $\operatorname{Sym}^p \cS_L$ are proved similarly, for which we will explain the required modifications at the end.
We will compute the cohomologies of the restriction of $\bwedge^p\cS_L$ to a fiber of $f: X_E \to X_{E\setminus n}$, and then apply Grauert's theorem.

A point $y\in X_{E\setminus n}$ is of the form $t\cdot p_{\mathscr F}$ for some $t\in (\kk^*)^{E\setminus n}$ and an ordered set partition $\mathscr F = (F_1, \ldots, F_\ell)$ of $E\setminus n$.
To reduce notational burden (such as $(t\cdot L)_{\mathscr F(i)}$), we assume without loss of generality that $t$ is the identity.
By Corollary~\ref{prop:constant}, the restriction $\bwedge^p \cS_L|_{f^{-1}(y)}$ is constant on all $\PP^1$-components of $f^{-1}(y)$ except possibly on $C(t,k)$ for some $1\leq k \leq \ell$ (if no such then fix an arbitrary $1\leq k \leq \ell$).  By Lemma~\ref{lem:Lbndry}, the restriction $\bwedge^p \cS_L|_{C(t,k)}$ to $C(t,k) \simeq \PP^1$ is isomorphic to $\bwedge^p\cS'_{L_{\mathscr F(k)}}$, where $\cS'_{L_{\mathscr F(k)}}$ is as defined in Definition~\ref{defn:tautP1}.
Hence we have $H^i(\bwedge^p \cS_L|_{f^{-1}(y)}) \simeq H^i(\bwedge^p\cS'_{L_{\mathscr F(k)}})$ for all $i$.
Now, Proposition~\ref{prop:wedgeH0} implies that $H^i(\bwedge^p\cS'_{L_{\mathscr F(k)}}) = 0$ for $i\geq 1$, and moreover, when we note that $L_{\mathscr F(j)}/n = (L/n)_{\mathscr F}$ for any $1\leq j \leq \ell$, the proposition gives
\[
H^0(\bwedge^p \cS'_{L_{\mathscr F(k)}}) \simeq \begin{cases}
\bwedge^p ((L/n)_{\mathscr F} \oplus \kk) & \text{if $n$ a coloop},\\
\bwedge^p((L/n)_{\mathscr F}) & \text{if $n$ not a coloop}.
\end{cases}
\]
In particular, as $\dim((L/n)_{\mathscr F}) = \dim (L/n)$, the dimension of $H^0(\bwedge^p(\cS_L|_{f^{-1}(y)})$ is invariant as the point $y$ varies.
The desired statements for $\bwedge^p \cS_L$ now follow from Grauert's theorem.

For the statements about $\bwedge^p\cQ_L$, $\operatorname{Sym}^p\cQ_L$, and $\operatorname{Sym}^p \cS_L$, one modifies the proof for $\bwedge^p\cS_L$ as follows.
For $\bwedge^p\cQ_L$, the proof is the same.
For $\operatorname{Sym}^p\cQ_L$, the proof is the same except that in the last step, Proposition~\ref{prop:symH} takes the place of Proposition~\ref{prop:wedgeH0}.
For $\operatorname{Sym}^p\cS_L$, again Proposition~\ref{prop:symH} is used in the last step, but it does not imply that $H^i(\operatorname{Sym}^p\cS'_{L_{\mathscr F(k)}}) = 0$ for $i\geq 1$, although it still implies the invariance of the dimension of $H^0(\operatorname{Sym}^p\cS_L|_{f^{-1}(y)})$.  Note the absence of statements about higher direct images of $\operatorname{Sym}^p\cS_L$ in Theorem~\ref{thm:symdelcont}.
\end{proof}

\begin{rem}
Let $\mathcal E$ be any globally generated vector bundle on $X_E$, such as $\bwedge^2\cQ_L \otimes \operatorname{Sym}^3\cQ_L$.
Its restriction to a fiber of $f$ has no higher cohomology since it is a globally generated bundle on a chain of $\PP^1$, and hence $f_*\mathcal E$ is a vector bundle with $R^i f_*\mathcal E  = 0$ for all $i>0$.  However, without sufficiently explicit description of $f_*\mathcal E$ like the one in Theorem~\ref{thm:extdelcont}, one cannot conclude much about $H^i(\mathcal E)$.  We currently do not have an analogue of Proposition~\ref{prop:wedgeH0} for arbitrary Schur/Weyl functors applied to $\cQ'_L$, and thus our treatment is restricted to exterior and symmetric powers.
\end{rem}

\section{Wonderful compactifications}\label{sec:wndcpt}

We begin with a review of wonderful compactifications introduced in \cite{dCP95}.
To avoid trivialities, throughout this section we assume $L\subseteq \kk^E$ to be loopless, so that the intersection $\PP L \cap \PP T$ is nonempty.

\smallskip
Let $\mathcal A$ be the arrangement of hyperplanes $\{\PP L \cap H_e : e\in E\}$ where $H_e$ is the $e$-th coordinate hyperplane of $\PP(\kk^E)$.  Notice that $\PP L \cap \PP T = \PP L \setminus (\bigcup \mathcal A)$.
Let $\mathcal P$ be the poset whose elements are the linear subvarieties $P\subseteq \PP L$ that arise as intersections of hyperplanes in $\mathcal A$, with partial ordering $G\leq G'$ given by reverse inclusion $P \supseteq P'$.  The poset $\mathcal P$ has the top and bottom elements $\hat 1 = \emptyset$ and $\hat 0 = \PP L$, respectively.
In matroid theory, this poset is known as the \emph{lattice of flats} of the matroid $\M$ of $L$.

\begin{defn}\label{defn:wndcpt}
A \emph{building set} $\mathcal G$ is a subset of $\mathcal P\setminus\{\hat 0, \hat 1\}$ such that for every $P \in \mathcal P\setminus \{\hat 0 , \hat 1\}$, the set $\max \mathcal G_{\leq P}$ of maximal elements of $\mathcal G$ in the interval $[\hat 0, P]$ satisfies
\[
[\hat 0, P] \simeq \prod_{G\in \max\mathcal G_{\leq P}} [\hat 0, G].
\]
The \emph{wonderful compactification} of $L$ with building set $\mathcal G$ is the variety $W_L^{\mathcal G}$ obtained from $\PP L$ by sequentially blowing-up the linear subvarieties of $\PP L$ in $\mathcal G$, starting with the smallest dimensional ones to the largest.
\end{defn}

The \emph{boundary} $\partial W_L^{\mathcal G} = W_L^{\mathcal G} \setminus (\PP L \cap \PP T)$ of $W_L^{\mathcal G}$ is a simple normal crossings divisor \cite[Section 3.1]{dCP95}.
A \emph{stratum} in the boundary is the intersection of a subset of the irreducible components of the boundary divisor $\partial W_L^{\mathcal G}$, which is necessarily smooth.
Note that $\mathcal P\setminus \{\hat 0, \hat 1\}$ itself is a building set, in which case we abuse notation to denote $W_L^{\mathcal P}=W_L^{\mathcal P \setminus \{\hat 0, \hat 1\}}$.
We now recall as a lemma two facts from the literature to prepare for the proof of Corollary~\ref{cor:wndcpt}, which stated that the log canonical divisor $K_{W_L^{\mathcal G}} + \partial W_L^{\mathcal G}$ has vanishing higher cohomology.

\pagebreak

\begin{lem}\label{lem:wndcpt}
Let notations be as above.
\begin{enumerate}
\item If $\mathcal G$ and $\mathcal H$ are building sets on $\mathcal P$ such that $\mathcal G \supseteq \mathcal H$, then there exists a sequence of building sets $(\mathcal G = \mathcal G_1, \mathcal G_2, \cdots, \mathcal G_\ell = \mathcal H)$ such that $W_L^{\mathcal G_i}$ is the blow-up of a stratum in the boundary of $W_L^{\mathcal G_{i+1}}$ for each $i = 1, \ldots, \ell-1$.
\item The variety $W_L^{\mathcal P}$ is isomorphic to the vanishing locus in $X_E$ of a global section of $\cQ_L$.  Under this isomorphism, we have $\mathcal O_{W_L^{\mathcal P}} (K_{W_L^{\mathcal P}} + \partial W_L^{\mathcal P}) \simeq \det \cQ_L|_{W_L^{\mathcal P}}$.
\end{enumerate}
\end{lem}

\begin{proof} (1) is a translation of \cite[Proof of Theorem 4.2]{FM05} into geometric language under the dictionary provided in \cite[Section 2]{dCP95} between the boundary strata structure of $W_L^{\mathcal G}$ and the simplicial complex known as the \emph{nested complex} of $\mathcal G$.

The first statement of (2) is \cite[Theorem 7.10]{BEST23}.
The second statement follows from \cite[Theorem 8.8]{BEST23}, which implies that $\det \mathcal T_{W_L^{\mathcal P}}(-\log \partial W_L^{\mathcal P}) \simeq \det \cS_L|_{W_L^{\mathcal P}}$, and $\det \cS_L^\vee \simeq \det \cQ_L$ from $0\to \cS_L \to \mathcal O_{X_E}^{\oplus E} \to \cQ_L \to 0$.
\end{proof}

\begin{proof}[Proof of Corollary~\ref{cor:wndcpt}]
We first claim that $H^i(\mathcal O_{W_L^{\mathcal G}}(K_{W_L^{\mathcal G}} + \partial W_L^{\mathcal G})) \simeq H^i(\mathcal O_{W_L^{\mathcal P}}(K_{W_L^{\mathcal P}} + \partial W_L^{\mathcal P}))$ for any building set $\mathcal G$.
Let $\pi: W_L^{\mathcal P} \to W_L^{\mathcal G}$ be the composition of blow-down maps given by Lemma~\ref{lem:wndcpt}(1).
Applying Lemma~\ref{lem:sncblowup} below to each of the blow-down maps making up $\pi$, we find $\pi^*(K_{W_L^{\mathcal G}} + \partial W_L^{\mathcal G}) = 
K_{W_L^{\mathcal P}} + \partial W_L^{\mathcal P}$.  
Lastly, since $\pi$ is a proper birational map of smooth varieties, by \cite[Theorem 2]{CR11} we have $\pi_* \mathcal O_{W_L^{\mathcal P}} = \mathcal O_{W_L^{\mathcal G}}$ and $R^i\pi_* \mathcal O_{W_L^{\mathcal P}} = 0$ for all $i>0$.\footnote{Because $\pi$ is a composition of blow-ups of along smooth loci, one may deduce this without \cite[Theorem 2]{CR11} by suitably modifying the proof of \cite[Proposition V.3.4]{Har77}, which is stated only for surfaces.}
Our claim now follows from the projection formula.

To finish, the first statement of Lemma~\ref{lem:wndcpt}(2) implies that we have the Koszul resolution
\[
0\to \det \cQ_L^\vee \to \cdots \to \bwedge^2 \cQ_L^\vee
\to \cQ_L^\vee \to \mathcal O_{X_E} \to \mathcal O_{W_L^{\mathcal P}} \to 0.
\]
Since $\det \mathcal E \otimes \bwedge^i \mathcal E^\vee \simeq \bwedge^{\operatorname{rank}(\mathcal E) - i} \mathcal E$ for a vector bundle $\mathcal E$, twisting the above resolution by $\det \cQ_L$ and noting the second statement of Lemma~\ref{lem:wndcpt}(2) gives the resolution
\[
0\to \mathcal O_{X_E}\to \cQ_L \to \bwedge^2 \cQ_L \to \cdots
\to  \det\cQ_L\to \mathcal O_{W_L^{\mathcal P}} (K_{W_L^{\mathcal P}} + \partial W_L^{\mathcal P})  \to 0.
\]
Applying Theorem~\ref{thm:extcohom} now yields the desired corollary by standard homological algebra \cite[Proposition B.1.2]{Laz04a}.
\end{proof}

\begin{lem}\label{lem:sncblowup}
Let $X$ be a smooth variety, and $D$ a simple normal crossings divisor with irreducible components $D_1, \dotsc, D_m$.
Let $Y$ be a subvariety of codimension $c$ that is a stratum in $D$, say $Y = D_1\cap \dotsb \cap D_c$.  Under the blow-up $\varphi: \widetilde X = \operatorname{Bl}_YX \to X$, let $E$ be the exceptional divisor, and $\widetilde D_i$ the strict transform of $D_i$.  Then, we have $\varphi^*(K_X + D) = K_{\widetilde X} + E + \widetilde D_1 + \dotsb + \widetilde D_m$.
\end{lem}

\begin{proof}
Recall that, for any blow-up $\varphi: \widetilde X \to X$ of a smooth $c$-codimensional subvariety $Y$ in a smooth variety $X$, we have $K_{\widetilde X} = \varphi^* K_X + (c-1) E$ \cite[Exercise II.8.5]{Har77}.
As $Y$ is contained in the (smooth) components $D_1, \dotsc, D_c$ as a smooth subvariety, and $Y$ intersects every other component of $D$ transversely (if nonempty), we find that $\varphi^*(D_i) = \widetilde D_i + E$ if $1\leq i\leq c$ and $\varphi^*(D_i) = \widetilde D_i$ if $i > c$.
Thus, we find $\varphi^*(K_X+D) = K_{\widetilde X} -(c-1)E + (\widetilde D_1 +E) + \dotsb + (\widetilde D_c +E) + \widetilde D_{c+1} + \dotsb + \widetilde D_m = K_{\widetilde X} + E + \widetilde D_1 + \dotsb + \widetilde D_m$.
\end{proof}

\begin{rem}
When the arrangement $\mathcal A$ from $L$ is the braid arrangement, consisting of hyperplanes $\{x \in \kk^{n-2}/\kk(1,\dotsc, 1) : x_i = x_j\}$ for $1\leq i < j \leq n-2$, the wonderful compactification with respect to the minimal building set is the moduli space $\overline{M}_{0,n}$ of stable rational curves with $n$ marked points.
Using the moduli structure of $\overline{M}_{0,n}$, \cite{Pan97} and \cite{Lee97} showed that any nonnegative sum of $\psi$-classes have vanishing higher cohomologies.
We do not know whether such techniques carry over to arbitrary wonderful compactifications to yield a strengthening of Corollary~\ref{cor:wndcpt}.
\end{rem}

\section{Questions}\label{sec:misc}

A broader theme behind Question~\ref{ques:troplin} is to ask: Which sheaf theoretic properties of realizations of matroids extend to all matroids?
We collect some related observations and questions.
We will now assume familiarity with matroid theory.
As in the previous section, we suppose $L\subseteq \kk^E$ to be loopless to avoid trivialities.

\smallskip
For an arbitrary not necessarily realizable matroid $\M$, there are $K$-classes $[\cS_\M]$ and $[\cQ_\M]$ in the Grothendieck $K$-ring of vector bundles on $X_E$ such that $[\cS_\M] = [\cS_L]$ and $[\cQ_\M] = [\cQ_L]$ whenever $\M$ has a realization $L\subseteq \kk^E$ \cite[Section 3.1]{BEST23}.
Let us denote by $D_{-P(\M)} = c_1(\cQ_\M)$ the first Chern class of $[\cQ_\M]$.
If $\M$ has a realization $L$, Lemma~\ref{lem:wndcpt}(2) states that the log canonical divisor of $W_L^{\mathcal P}$ is $D_{-P(\M)}|_{W_L^{\mathcal P}}$.
Even if $\M$ is not realizable, we may consider the line bundle $\mathcal O_{X_E}(D_{-P(\M)})$.

\begin{rem}\label{rem:basepolytope}
Let us sketch an explanation of the notation $D_{-P(\M)}$.
The \emph{matroid base polytope} of a matroid $\M$ on a ground set $E$ is the polytope $P(\M)  = \text{the convex hull of }\{\be_B \mid B \text{ a basis of }\M\}$ in $\RR^E$.
Let $-P(\M)$ be the polytope $\{x \in \RR^E : -x \in P(\M)\}$.
Under a known correspondence between nef divisors on a toric variety and certain polytopes \cite[Chapter 6]{CLS11}, the polytope $-P(\M)$ defines a nef divisor $D_{-P(\M)}$ on the permutohedral variety $X_E$, which can be shown to equal $c_1(\mathcal Q_\M)$.
See \cite[Section 2.7 and Appendix III]{BEST23} for details, including conventions that explain the minus $-P(\M)$ (instead of $P(\M)$).
\end{rem}

\subsection{Immaculate line bundles}
We have the following variation of Corollary~\ref{cor:wndcpt}.

\begin{cor}
Let $L'\subseteq L$ be a subspace such that $\dim L  = \dim L' - 1$, and let $\M'$ be the matroid of $L'$.  Then, the line bundle $\mathcal O_{W_L^{\mathcal P}}(D_{-P(\M)} - D_{-P(\M')})$ on $W_L^{\mathcal P}$ satisfies
\begin{align*}
H^i\big(\mathcal O_{W_L^{\mathcal P}}(D_{-P(\M)} - D_{-P(\M')})\big) = 0 \quad\text{for all $i>0$,}\quad\text{and}\\
\dim H^0\big(\mathcal O_{W_L^{\mathcal P}}(D_{-P(\M)} - D_{-P(\M')})\big) = \begin{cases} 1 & \text{if $\M'$ has loops}\\ 0 & \text{if $\M'$ is loopless}.\end{cases}
\end{align*}
In particular, the line bundle $\mathcal O_{W_L^{\mathcal P}}(D_{-P(\M)} - D_{-P(\M')})$ on $W_L^{\mathcal P}$ is \emph{immaculate}, i.e.\ has no nonzero cohomologies, if $\M'$ is loopless.
\end{cor}

\begin{proof}
By construction, from $L'\subset L \subseteq \kk^E$ we have the surjective map $\cQ_{L'} \to \cQ_{L}$.
Let $\mathcal L_{\M',\M}$ be the kernel, so that we have the short exact sequence
\[
0 \to \mathcal L_{\M',\M} \to \cQ_{L'} \to \cQ_L \to 0.
\]
By taking $\det$ of the sequence, we see that $\mathcal L_{\M',\M} \simeq \mathcal O_{X_E}(D_{-P(\M')} - D_{-P(\M)})$.  Applying duality and exterior power, we obtain for each $p\geq 1$ a short exact sequence
\[
0 \to \bwedge^p \cQ_L^{\vee} \to \bwedge^p \cQ_{L'}^{\vee} \to \bwedge^{p-1} \cQ_L^{\vee} \otimes \mathcal L^\vee_{\M',\M} \to 0.
\]
Applying Theorem~\ref{thm:extcohom} and Remark~\ref{rem:crem} to the long exact sequence of cohomologies, we thus obtain $H^i(\bwedge^{p-1} \cQ_L^{\vee} \otimes \mathcal L^\vee_{\M',\M}) = 0$ for all $i>0$ and $p \geq 1$.  Moreover, we have
\[
\dim H^0(\bwedge^{p-1} \cQ_L^{\vee} \otimes \mathcal L^\vee_{\M',\M}) = \dim H^0(\bwedge^p \cQ_{L'}^\vee) -  \dim H^0(\bwedge^p \cQ_{L}^\vee) =  \textstyle\binom{|\text{loops}(\M')|}{p} -   \binom{|\text{loops}(\M)|}{p},
\]
where for the last equality we used that $\operatorname{crem}\cQ_L^\vee \simeq \cS_{L^\perp}$, and $L^\perp$ realizes the dual matroid $\M^\perp$ whose coloops correspond to the loops of the original matroid $\M$.  To finish, recall the Koszul resolution $\bwedge^\bullet \cQ_L^\vee \to \mathcal O_{W_L^{\mathcal P}} \to 0$ from the proof of Corollary~\ref{cor:wndcpt}.  Twisting the resolution by $\mathcal L^\vee_{\M',\M}$ and taking cohomology, keeping in mind standard homological algebra \cite[Proposition B.1.2]{Laz04a}, one obtains the desired result.
\end{proof}

An \emph{elementary matroid quotient} $\M \twoheadrightarrow \M'$ consists of two matroids $\M$ and $\M'$ whose ranks differ by 1 such that every flat of $\M'$ is a flat of $\M$.  It is \emph{realizable} if there is a flag of linear subspaces $L' \subseteq L \subseteq \kk^E$ such that $L'$ and $L$ respectively realize $\M'$ and $\M$.
In light of the corollary above, we ask the following question.

\begin{ques}
For any elementary matroid quotient $\M\twoheadrightarrow \M'$, not necessarily realizable, is the line bundle $\mathcal O_{W_L^{\mathcal P}}(D_{-P(\M)} - D_{-P(\M')})$ on $W_L^{\mathcal P}$ immaculate, i.e.\ has no nonzero cohomologies, if $\M'$ is loopless?

Is there a theory of tropical line bundles and their sheaf cohomology on tropical manifolds such that it agrees with the above corollary?
\end{ques}

\subsection{Log canonical image}
We conclude with a discussion of the log canonical image of a wonderful compactification of $L$.
The line bundle $\det \cQ_L \simeq \mathcal O_{X_E}(D_{-P(\M)})$ is globally generated, with torus-invariant sections in bijection with the bases of $\M$ (see \cite[Proposition 4.3.3]{CLS11} and \cite[Section 2.7 and Example 3.11]{BEST23}).
We may thus consider the embedded projective variety
\begin{align*}
X_L & = \text{the closure of the image of $\PP L \cap \PP T$ under the map $\varphi: X_E \to \PP(H^0(\det \cQ_L))$}\\
& = \text{the closure of $\PP L \cap \PP T$ in the toric variety of the matroid polytope $-P(\M)$ (Remark~\ref{rem:basepolytope}).}
\end{align*}
This variety $X_L$ is also known as Kapranov's \emph{visible contour} \cite{Kap93}.  When $\M$ is connected, as we shall assume from now, the variety $X_L$ is the log canonical model of $\PP L \cap \PP T$ with (\'etale locally) toric singularities \cite[Section 2]{HKT06}.  For a building set $\mathcal G$, the map $W_L^{\mathcal G} \to X_L$ given by the log canonical bundle of $W_L^{\mathcal G}$ is a (\'etale locally) toric resolution of singularities, 
and thus $H^i(\mathcal O_{X_L}(\ell)) \simeq H^i(\det \cQ_L|_{W_L^{\mathcal P}}^{\otimes \ell})$ for all $\ell \in \ZZ$.
In particular, applying Theorem~\ref{thm:extcohom} and Remark~\ref{rem:crem} to the Koszul complexes in the proof of Corollary~\ref{cor:wndcpt} yields the following.

\begin{cor}\label{cor:idealsheaf}
The ideal sheaf $\mathcal I_{X_L}$ satisfies
\[
H^i(\mathcal I_{X_L}) = 0 \quad\text{and}\quad H^i(\mathcal I_{X_L}(1)) = 0 \quad\text{for all $i>0$},
\]
and hence $H^i(\mathcal O_{X_L}) =0$ and $H^i(\mathcal O_{X_L}(1)) = 0$ for all $i>0$.
\end{cor}

Over characteristic zero, applying \cite[Theorem C]{BF22} further implies that $H^i(\mathcal I_{X_L}(\ell)) = 0$ and $H^i(\mathcal O_{X_L}(\ell)) = 0$ for all $i>0$ and $\ell \geq 0$.
Moreover, over characteristic zero, by Kawamata--Viehweg vanishing we have $H^i(\mathcal O_{X_L}(-\ell)) = 0$ for all $i< \dim X_L$ and $\ell >0$.
We thus ask the following, part of which is a strengthening of Speyer's question in Remark~\ref{rem:speyer}.

\begin{ques}
Suppose $\kk$ has positive characteristic.  Is $H^i(\mathcal I_{X_L}(\ell)) = 0$ and $H^i(\mathcal O_{X_L}(\ell))=0$ for all $i>0$ and $\ell \geq 0$?  Is $H^i(\mathcal O_{X_L}(-\ell)) = 0$ for all $i<\dim X_L$ and $\ell >0$?  In particular, is the embedded variety $X_L$ projectively normal and/or arithmetically Cohen-Macaulay?\footnote{Matt Larson also proposed this question during the Banff workshop ``Algebraic Aspects of Matroid Theory.''}
\end{ques}


\begin{rem}
Given a fixed total order of $E$, one can show that the restrictions to $W_L^{\mathcal P}$ of the torus-invariant sections of $\det \cQ_L$ are spanned by those that correspond to the \emph{nbc-bases} of the matroid $\M$.
One can moreover show that they not only span but also form a basis of $H^0(\mathcal O_{X_L}(1))$, by using Corollary~\ref{cor:idealsheaf} and by noting that the quantity
\[
\dim H^0(\mathcal O_{X_L}(1)) = \dim H^0 (\mathcal O_{W_L^{\mathcal P}} (K_{W_L^{\mathcal P}} + \partial W_L^{\mathcal P}) ) = \sum_{\substack{S\subseteq E \\ \text{$S$ contains }\\ \text{a basis of $\M$}}}(-1)^{|S|-r}
\]
is the \emph{M\"obius invariant} $\mathrm{T}_\M(1,0)$ of $\M$, which equals the number of nbc-bases of $\M$ \cite{Bry77}.
\end{rem}

\small
\bibliography{Eur_CohomTautMat.bib}
\bibliographystyle{alpha}

\end{document}